\documentclass{llncs}
\usepackage{amsmath,amssymb,amsbsy,amsfonts,latexsym,
               amsopn,amstext,amsxtra,euscript,amscd}

\newtheorem{thm}{Theorem}
\newtheorem{lem}[thm]{Lemma}
\newtheorem{cor}[thm]{Corollary}

\newtheorem{rem}[thm]{Remark}

\def\cB{{\mathcal B}}
\def\cC{{\mathcal C}}

\def\cH{{\mathcal H}}

\def\cGB{\mathcal{GB}}

\def\e{{\mathbf e}}

\def\uu{{\bf u}}

\def\xx{{\bf x}}
\def\yy{{\bf y}}
\def\zz{{\bf z}}

\def\00{{\bf 0}}
\def\11{{\bf 1}}
\def\+{\oplus}

\def\\{\cr}
\def\({\left(}
\def\){\right)}

\newcommand{\BBZ}{\mathbb{Z}}
\newcommand{\BBR}{\mathbb{R}}

\newcommand{\BBC}{\mathbb{C}}

\newcommand{\BBQ}{\mathbb{Q}}

\providecommand{\newoperator}[3]{%
  \newcommand*{#1}{\mathop{#2}#3}}

\newoperator{\FD}{\mathrm{FD}}{\nolimits}

\begin{document}
\title{Octal Generalized Boolean Functions}
\author{
Pantelimon~St\u anic\u a,
Thor Martinsen,
}
\institute{
$^1$ Department of Applied Mathematics\footnote{T.M. is a Ph.D student in Applied Mathematics at the Naval Postgraduate School.}\\
Naval Postgraduate School\\
Monterey, CA 93943--5216, USA\\
\email{\{tmartins,pstanica\}@nps.edu}
}
\date{\today}
\maketitle
\begin{abstract}
In this paper we characterize (octal) bent generalized Boolean functions defined on $\BBZ_2^n$ with values in $\BBZ_8$.
Moreover, we propose several constructions of such generalized bent functions for both $n$ even and $n$ odd.
\end{abstract}

\section{Introduction}

Several generalizations of Boolean functions have been proposed in the
recent years and the effect of the Walsh--Hadamard transform on them has been
studied~\cite{KSW85,KAI07,ST09}. The natural generalizations
of bent functions in the Boolean case, namely generalized functions which have flat spectra
with respect to the Walsh--Hadamard
transform are of special interest.

Let the set of integers, real numbers and complex numbers
be denoted by $\BBZ$, $\BBR$ and $\BBC$, respectively. By $\BBZ_r$ we denote the ring of integers modulo $r$.
A function from $\BBZ_2^n$ to $\BBZ_2$ is said to be
a Boolean function on $n$ variables and the set of all such functions is denoted by
$\cB_n$. A function from $\BBZ_2^n$ to
$\BBZ_q$ ($q$ a positive integer) is said to be a {\em generalized Boolean function} on $n$
variables~\cite{ST09}. We denote the set of such functions by
  $\cGB_n^q$. We will consider these functions with an emphasis on $q=8$.

Any element $\xx \in \BBZ_2^n$ can be written as
an $n$-tuple
$(x_n, \ldots, x_1)$, where $x_i \in \BBZ_2$ for all $i = 1, \ldots, n$.
The addition over $\mathbb{Z}$, $\mathbb{R}$ and $\mathbb{C}$ is
denoted by `$+$'. The addition over
$\BBZ_2^n$ for all $n\geq 1$, is denoted by  $\oplus$.
Addition modulo $q$ is denoted by `$+$' and is understood from the context.
If $\xx=(x_n, \ldots, x_1)$ and $\yy=(y_n, \ldots, y_1)$ are two elements of $\BBZ_2^n$,
we define the scalar (or  inner) product,
by
$
\xx \cdot \yy = x_ny_n\oplus\cdots\oplus  x_2y_2 \oplus x_1y_1.
$
The cardinality of the set $S$ is denoted by $|S|$.  If $z=a+b\,\imath \in \mathbb{C}$, then
$|z|=\sqrt{a^2+b^2}$ denotes the absolute value of $z$, and $\overline{z}=a-b\,\imath$ denotes
the complex conjugate of $z$, where $\imath^2=-1$, and $a,b\in\BBR$. The conjugate of a bit $b$ will also be denoted by $\bar b$.

The (normalized) {\em Walsh--Hadamard transform} of $f \in \cB_n$ at
any point $\uu \in \BBZ_2^n$ is defined by
\[
W_f(\uu) = 2^{-\frac{n}{2}} \sum_{\xx \in \BBZ_2^n} (-1) ^{f(\xx)} (-1)^{\uu \cdot \xx}.
\]
A function $f \in \cB_n$, where $n$ is even, is a {\em bent function}  if
$|W_f(\uu)|=1$ for all $\uu \in \BBZ_2^n$. In case $n$ is even a function $f \in \cB_n$ is said to be
a {\em semibent} function if and only if, $|W_f(\uu)| \in \{ 0, \sqrt{2}\}$
for all $\uu \in \BBZ_2^n$.

{\begin{sloppypar}
The sum $$C_{f,g}(\zz)=\sum_{\xx \in \BBZ_2^n} (-1) ^{f(\xx)  \oplus g(\xx \+ \zz)}$$
is  the {\em crosscorrelation} of $f$ and
$g$ at $\zz$.
The {\em autocorrelation} of $f \in \cB_n$ at $\uu \in \BBZ_2^n$
is $C_{f,f}(\uu)$ above, which we denote by $C_f(\uu)$.
\end{sloppypar}
}

Let $\zeta=e^{2\pi \imath/q}$ be the $q$-primitive root of unity.
The (generalized) {\em Walsh--Hadamard transform} of $f \in \cGB_n^q$ at
any point $\uu \in \BBZ_2^n$ is the complex valued function defined by
\[
\cH_f(\uu) = 2^{-\frac{n}{2}} \sum_{\xx \in \BBZ_2^n} \zeta^{f(\xx)} (-1)^{\uu \cdot \xx}.
\]
A function $f \in \mathcal{B}_n$ is a {\em generalized bent function} ({\em gbent}, for short) if
$|\cH_f(\uu)|=1$ for all $\uu \in \BBZ_2^n$.

{\begin{sloppypar}
The sum $$\cC_{f,g}(\zz)=\sum_{\xx \in \BBZ_2^n} \zeta ^{f(\xx)  - g(\xx \+ \zz)}$$
is  the {\em crosscorrelation} of $f$ and
$g$ at $\zz$.
The {\em autocorrelation} of $f \in \cGB_n^q$ at $\uu \in \BBZ_2^n$
is $\cC_{f,f}(\uu)$ above, which we denote by $\cC_f(\uu)$.
\end{sloppypar}
}

When $2^{h-1}<q\leq 2^h$, given any $f \in \cGB_n^q$ we associate a unique sequence of Boolean
 functions $a_i\in \cB_n$ ($i=0,1,\ldots,h-1$) such that
\begin{equation}
\label{eq0.1}
f(\xx) = a_0(\xx) + 2 a_1(\xx)+\cdots+2^{h-1} a_{h-1}(\xx), \mbox{ for all } \xx \in \BBZ_2^n.
\end{equation}

\section{Properties of Walsh--Hadamard transform on generalized Boolean functions }
\label{propnega}

 We gather in the current section several properties of the Walsh--Hadamard transform and its generalized counterpart~\cite{GSS12}.

\begin{thm} We have:
\begin{itemize}
\item[$(i)$]
Let $f \in \cB_n$. Then, the inverse of the Walsh--Hadamard transform is
\begin{equation*}
\label{eq1.00}
(-1) ^{f(\yy)} = 2^{-\frac{n}{2}} \sum_{\uu \in \BBZ_2^n} W_f(\uu)(-1)^{\uu \cdot \yy}.
\end{equation*}
\item[$(ii)$]
If $f,g\in \cB_n$, then
\begin{equation*}
\label{eq3.00}
\begin{split}
\sum_{\uu \in \BBZ_2^n} C_{f,g}(\uu)(-1)^{\uu \cdot \xx} = 2^n W_f(\xx)W_g(\xx), \\
C_{f,g}(\uu) = \sum_{\xx \in \BBZ_2^n}W_f(\xx)W_g(\xx) (-1)^{\uu \cdot \xx}.
\end{split}
\end{equation*}
\item[$(iii)$]
Taking the particular case $f = g$ we obtain
\begin{equation*}
\label{eq4.00}
C_f(\uu) = \sum_{\xx \in \BBZ_2^n}W_f(\xx)^2 (-1)^{\uu \cdot \xx}.
\end{equation*}
\item[$(iv)$]
A Boolean function $f$ is  bent if and only if
$C_f(\uu) = 0$ at all nonzero points $\uu \in \BBZ_2^n$.
\item[$(v)$]
For any $f \in \cB_n$, the Parseval's identity holds
\begin{equation*}
\label{eq6.00}
\sum_{\xx \in \BBZ_2^n}W_f(\xx)^2 = 2^n.
\end{equation*}
\end{itemize}
\end{thm}
For more properties of these transforms
and Boolean functions, the interested reader can consult~\cite{CH1,CH2,PS09}.

The properties of the Walsh--Hadamard transform on generalized Boolean functions are
similar to the Boolean function case.
\begin{thm}
We have:
\begin{itemize}
\item[$(i)$] Let $f \in \cGB_n^q$.
The inverse of the Walsh--Hadamard transform is given by
\begin{equation*}
\label{eq1}
\zeta ^{f(\yy)} = 2^{-\frac{n}{2}} \sum_{\uu \in \BBZ_2^n} \cH_f(\uu)(-1)^{\uu \cdot \yy}.
\end{equation*}
Further, $\cC_{f,g}(\uu) = \overline{\cC_{g,f}(\uu)}$, for all $\uu \in \BBZ_2^n$, which implies that $\cC_f(\uu)$ is always real.
\item[$(ii)$] If $f,g\in\cGB_n^q$, then
\begin{equation*}
\label{eq3}
\begin{split}
\sum_{\uu \in \BBZ_2^n}\cC_{f,g}(\uu)(-1)^{\uu \cdot \xx} = 2^n \cH_f(\xx)\overline{\cH_g(\xx)}, \\
\cC_{f,g}(\uu) = \sum_{\xx \in \BBZ_2^n}\cH_f(\xx)\overline{\cH_g(\xx)} (-1)^{\uu \cdot \xx}.
\end{split}
\end{equation*}
\item[$(iii)$] Taking the particular case $f = g$ we obtain
\begin{equation}
\label{eq4}
\cC_f(\uu) = \sum_{\xx \in \BBZ_2^n}|\cH_f(\xx)|^2 (-1)^{\uu \cdot \xx}.
\end{equation}
\item[$(iv)$]
If $f\in\cGB_n^q$, then $f$ is gbent if and only if
\begin{equation}
\label{eq5}
\cC_f(\uu) = \begin{cases}
2^n & \mbox{ if } \uu = 0,\\
0 & \mbox{ if } \uu \ne 0.
\end{cases}
\end{equation}
\item[$(v)$]
Moreover, the (generalized) Parseval's identity holds
\begin{equation}
\label{eq6}
\sum_{\xx \in \BBZ_2^n}|\cH_f(\xx)|^2 = 2^n.
\end{equation}
\end{itemize}
\end{thm}

\section{The Walsh--Hadamard Transform on components}

Let $\zeta=\e^{2\pi\imath/q}$ be  a $q$-primitive root of unity.
Let $f$ be written as $f(\xx)=\sum_{i=0}^{h-1} a_i(\xx) 2^i$.
For brevity, we use the notations $\zeta_i:=\zeta^{2^i}$.
It is easy to see that, for $s\in\BBZ_2$, we have
\begin{equation}
\label{eq-zs}
z^s=\frac{1+(-1)^s}{2}+\frac{1-(-1)^s}{2}z,
\end{equation}
 and so, we have the identities
$\zeta_i^{a_i(\xx)}=\frac{1}{2}\left(A_i+A_i'\zeta_i \right)$,
where $A_i=1+(-1)^{a_i(\xx)}$ and $A_i'=1-(-1)^{a_i(\xx)}$, $\bar I=\{0,1,\ldots,h-1\}\setminus I$.

The Walsh--Hadamard coefficients of $f$ are
\begin{eqnarray*}
&&2^{n/2} \cH_f(\uu)
=\sum_\xx \zeta^{f(\xx)}(-1)^{\uu\cdot \xx}=\sum_\xx \zeta^{\sum_{i=0}^{h-1} a_i(\xx) 2^i}(-1)^{\uu\cdot \xx}\\
&=& \sum_\xx  (-1)^{\uu\cdot \xx} \prod_{i=0}^{h-1} \left(\zeta^{2^i} \right)^{a_i(\xx)}\\
&=&\sum_\xx  (-1)^{\uu\cdot \xx} \prod_{i=0}^{h-1} \frac{1}{2}\left(1+(-1)^{a_i(\xx)}+(1-(-1)^{a_i(\xx)})\zeta_i \right)\\
&=& 2^{-h} \sum_\xx  (-1)^{\uu\cdot \xx} \sum_{I\subseteq \{0,\ldots,h-1\}}
 \prod_{i\in I,j \in \bar I} \zeta_i A_i' A_j\\
 &=& 2^{-h} \sum_\xx  (-1)^{\uu\cdot \xx} \sum_{I\subseteq \{0,\ldots,h-1\}}\zeta^{\sum_{i\in I} 2^i}
 \prod_{i\in I,j \in\bar I}  A_i' A_j\\
 &=& 2^{-h} \sum_\xx  (-1)^{\uu\cdot \xx} \sum_{I\subseteq \{0,\ldots,h-1\}} \zeta^{\sum_{i\in I} 2^i}
 \sum_{J\subseteq I,K\subseteq \bar I}
  (-1)^{|J|} (-1)^{\sum_{j\in J}a_j(\xx)\+\sum_{k\in K}a_k(\xx)}\\
 &=& 2^{-h} \sum_{I\subseteq \{0,\ldots,h-1\}} \zeta^{\sum_{i\in I} 2^i}\sum_{J\subseteq I,K\subseteq \bar I}
  (-1)^{|J|} \sum_\xx  (-1)^{\uu\cdot \xx} (-1)^{\sum_{\ell\in J\cup K}a_\ell(\xx)},
\end{eqnarray*}
and so, we obtain the next result.
\begin{thm}
The Walsh--Hadamard transform of $f:\BBZ_2^n\to \BBZ_q$, $2^{h-1}<q\leq 2^h$,
where $f(\xx)=\sum_{i=0}^{h-1} a_i(\xx) 2^i$, $a_i\in\cB_n$ is given by
\begin{eqnarray*}
\cH_f(\uu)= 2^{-h} \sum_{I\subseteq \{0,\ldots,h-1\}} \zeta^{\sum_{i\in I} 2^i}\sum_{J\subseteq I,K\subseteq \bar I}
  (-1)^{|J|} W_{\sum_{\ell\in J\cup K}a_\ell(\xx)}(\uu).
\end{eqnarray*}
\end{thm}

In the next section we will redo some of these calculations, for the particular case $q=8$, which will
allow us to completely describe the generalized bent Boolean functions in that case.

\section{A Characterization of Generalized Bent Functions in $\BBZ_8$}

In this section we extend the result of Sol\'e and Tokareva~\cite{ST09} to generalized Boolean functions
from $\BBZ_2^n$ into $\BBZ_8$.
Let $\zeta=\e^{2\pi\imath/8}=\frac{\sqrt{2}}{2}(1+\imath)$ be  the 8-primitive root of unity.
  Every  function $f:\BBZ_2^n\to \BBZ_8$ can be written as
\begin{equation}
\label{f_z8}
f(\xx)=a_0(\xx)+a_1(\xx)2+a_2(\xx) 2^2,
\end{equation}
where $a_i(\xx)$ are Boolean functions, and $`+'$ is the addition modulo 8.
We prove the next lemma, which gives the connection between Walsh--Hadamard transforms of $f$ and it components
as in~\eqref{f_z8}.
\begin{lem}
Let $f\in\cGB_n^8$ as in \eqref{f_z8}. Then,
\[
4\cH_f(\uu)=\alpha_0 W_{a_2}(\uu)+\alpha_1 W_{a_0\oplus a_2}(\uu)+\alpha_2 W_{a_1\oplus a_2}(\uu)+\alpha_3 W_{a_0\oplus a_1\oplus a_2}(\uu),
\]
where $\alpha_0=1+(1+\sqrt{2})\imath$, $\alpha_1=1+(1-\sqrt{2})\imath$, $\alpha_2=1+\sqrt{2}-\imath$, $\alpha_3=1-\sqrt{2}-\imath$.
\end{lem}

\begin{proof}
We compute
\begin{eqnarray}
\label{eq_hf}
2^{n/2} \cH_f(\uu)&=&\sum_{\xx\in\BBZ_2^n} \zeta^{f(\xx)} (-1)^{\uu\cdot \xx}\nonumber\\
&=&\sum_{\xx\in\BBZ_2^n} \zeta^{a_0(\xx)+a_a(\xx)2+a_2(\xx) 2^2} (-1)^{\uu\cdot \xx}\\
&=& \sum_{\xx\in\BBZ_2^n} \zeta^{a_0(\xx)} \imath^{a_1(\xx)}  (-1)^{a_2(\xx)\+\uu\cdot \xx}.\nonumber
\end{eqnarray}
Use  formula~\eqref{eq-zs} with $z=\imath$ and $z=\zeta$ in equation \eqref{eq_hf}, and obtain
\begin{eqnarray*}
\label{eq_hf2}
2^{n/2} \cH_f(\uu)&=&\sum_{\xx\in\BBZ_2^n} \left( \frac{1+(-1)^{a_0(\xx)}}{2}+\frac{1-(-1)^{a_0(\xx)}}{2}\zeta \right)\\
&& \qquad\cdot \left(  \frac{1+(-1)^{a_1(\xx)}}{2}+\frac{1-(-1)^{a_1(\xx)}}{2}\imath \right)(-1)^{a_2(\xx)\+\uu\cdot \xx}\\
&=& \frac{1}{4}  \sum_{\xx\in\BBZ_2^n}(-1)^{a_2(\xx)\+\uu\cdot \xx}
\left(
1+(1+\sqrt{2})\imath+(1+(1-\sqrt{2})\imath)(-1)^{a_0(\xx)}\right.\\
&& \quad\left.+(1+\sqrt{2}-\imath)(-1)^{a_1(\xx)}+
(1-\sqrt{2}-\imath)(-1)^{a_0(\xx)}(-1)^{a_1(\xx)}
   \right)\\
&=& \frac{1}{4}  \sum_{\xx\in\BBZ_2^n}
\left(
\alpha_0 (-1)^{a_2(\xx)\+\uu\cdot \xx} +\alpha_1 (-1)^{a_0(\xx)\+a_2(\xx)\+\uu\cdot \xx}\right.\\
&&\qquad\left. +
\alpha_2 (-1)^{a_1(\xx)\+a_2(\xx)\+\uu\cdot \xx} +\alpha_3(-1)^{a_0(\xx)\+a_1(\xx)\+a_2(\xx)\+\uu\cdot \xx}
\right),
\end{eqnarray*}
from which we derive our result.
\qed
\end{proof}

\begin{cor}
\label{cor-norm}
With the notations of the previous lemma, we have
\begin{equation}
\label{eq-norm}
4\sqrt{2}|\cH_f(\uu)|^2=W^2-X^2+2XY+Y^2-2WZ-Z^2+\sqrt{2}(W^2+X^2+Y^2+Z^2),
\end{equation}
where, we use for brevity, $W:=W_{a_2}(\uu)$, $X:=W_{a_0\oplus a_2}(\uu)$, $Y:= W_{a_1\oplus a_2}(\uu)$,
$Z:=W_{a_0\oplus a_1\oplus a_2}(\uu)$.
\end{cor}
\begin{proof}
By replacing $\alpha_i,\zeta$ by their complex representations, the corollary follows in
 a rather straightforward, albeit tedious manner.
 \qed
\end{proof}
\begin{thm}
Let $f\in\cGB_n^8$ as in \eqref{f_z8}. Then:
\begin{enumerate}
\item[$(i)$] If $n$ is even, then $f$ is generalized bent if and only if
 $a_2,a_0\oplus a_2,a_1\oplus a_2,a_0\oplus a_1\oplus a_2$ are all bent, and
 $(*)$ $W_{a_0\oplus a_2}(\uu) W_{a_1\oplus a_2}(\uu)=W_{a_2}(\uu)W_{a_0\oplus a_1\oplus a_2}(\uu)$, for all $\uu\in\BBZ_2^n$;
 \item[$(ii)$] If $n$ is odd, then $f$ is generalized bent if and only if
$a_2,a_0\oplus a_2,a_1\oplus a_2,a_0\oplus a_1\oplus a_2$ are semi-bent with their values satisfying $(*)$.
\end{enumerate}
\end{thm}
\begin{proof}
We use the $W,X,Y,Z$ notations of Corollary~\ref{cor-norm}.
First, assume that $a_2,a_0\oplus a_2,a_1\oplus a_2,a_0\oplus a_1\oplus a_2$ are all bent (respectively, semi-bent).
  Then, replacing the corresponding values of the Walsh--Hadamard transforms in equation~\eqref{eq-norm}
  (and using the imposed condition~(*) on the Walsh--Hadamard coefficients) we obtain
\[
4\sqrt{2}|\cH_f(\uu)|^2=4\sqrt{2},
\]
and so, $|\cH_f(\uu)|=1$, that is, $f$ is gbent.

Conversely, we assume that $f$ is gbent, and so,
\[
4\sqrt{2}=W^2-X^2+2XY+Y^2-2WZ-Z^2+\sqrt{2}(W^2+X^2+Y^2+Z^2),
\]
which prompts the system
\begin{eqnarray}
&&W^2-X^2+2XY+Y^2-2WZ-Z^2=0\label{sys1}\\
&&W^2+X^2+Y^2+Z^2=4.\label{sys2}
\end{eqnarray}
We are looking for solutions in  $2^{-n/2}\,\BBZ$ (a subset of $\BBQ$, if $n$ is even or $\sqrt{2}\,\BBQ$, if $n$ is odd).

 We look at equation \eqref{sys2}, initially, and apply Jacobi's four squares theorem (see \cite{Hi87}, for instance). \newline
 {\em Case $(i)$.} Let $n=2k$ be even.
  Thus, $W,X,Y,Z$ are all rational (certainly, not all 0).
  Write $W=2^{-n/2}W',X=2^{-n/2}X',Y=2^{-n/2}Y',Z=2^{-n/2}Z'$, and replace \eqref{sys1} and \eqref{sys2} by
  the system in integers
  \begin{eqnarray}
&&W'^2-X'^2+2X'Y'+Y'^2-2W'Z'-Z'^2=0\label{sys1-int}\\
&&W'^2+X'^2+Y'^2+Z'^2=2^{2k+2}.\label{sys2-int}
\end{eqnarray}
Now, by  Jacobi's four-squares theorem, we know there are exactly 24 solutions of \eqref{sys2-int},
 which are all variations in $\pm$ sign and order of
$(\pm 2^k,\pm 2^k,\pm 2^k, \pm 2^k)$ or $(\pm 2^{k+1}, 0,0,0)$. Further, it is straightforward to check
that  among these 24 solutions, only the  eight tuples $(X',Y',W',Z')$ in the list below
are also satisfying equation \eqref{sys1-int},
\begin{equation*}
\label{sol1}
\begin{split}
&(-2^k,  -2^k,  -2^k,  -2^k), (2^k,  2^k,  -2^k,  -2^k), (-2^k,  -2^k,  2^k,  2^k), (-2^k,  2^k,  -2^k,  2^k), \\
& (2^k,  -2^k,  -2^k,  2^k), (-2^k,  2^k,  2^k,  -2^k), (2^k,  -2^k,  2^k,  -2^k),(2^k,  2^k,  2^k,   2^k).
\end{split}
\end{equation*}
This  implies  that $(X,Y,W,Z)\in 2^{-n/2}\BBZ^4$ are any of
the following:
\begin{equation}
\label{sol-bent}
\begin{split}
&(-1,  -1,  -1,  -1), (1,  1,  -1,  -1), (-1,  -1,  1,  1), (-1,  1,  -1,  1), \\
& (1,  -1,  -1,  1), (-1,  1,  1,  -1), (1,  -1,  1,  -1),(1,  1,  1,   1),
\end{split}
\end{equation}
and $(i)$ is shown  (one can check easily that these solutions also satisfy condition~$(*)$).\newline
 {\em Case $(ii)$.} Let $n=2k+1$ be odd. Then, at least one of $X,Y,W,Z$ is nonzero and belongs to $\sqrt{2}\,\BBQ$).
  As before, write $W=2^{-n/2}W',X=2^{-n/2}X',Y=2^{-n/2}Y',Z=2^{-n/2}Z'$, and replace \eqref{sys1} and \eqref{sys2} by
  the system in integers
  \begin{eqnarray}
&&W'^2-X'^2+2X'Y'+Y'^2-2W'Z'-Z'^2=0\label{sys1-sqrt}\\
&&W'^2+X'^2+Y'^2+Z'^2=2\cdot 2^{2k+2},\label{sys2-sqrt}
\end{eqnarray}
and so, by  Jacobi's four-squares theorem, equation \eqref{sys2-sqrt} has exactly 24 solutions,
 which are all variations in $\pm$ sign and order of
$(\pm 2^{k+1},\pm 2^{k+1},0, 0)$. Further, it is straightforward to check
that  among these 24 solutions, the  eight tuples $(X',Y',W',Z')$ in the list below
are also satisfying equation \eqref{sys1-sqrt},
\begin{equation*}
\label{sol-semibent-all}
\begin{split}
& ( 0, 2^{k+1},  0, 2^{k+1}), ( 0, 2^{k+1},  0, -2^{k+1}), ( 0, -2^{k+1},  0, 2^{k+1}),( 0, -2^{k+1},  0,-2^{k+1})\\
& (2^{k+1},  0, 2^{k+1},0), ( 2^{k+1},  0, -2^{k+1},0,(-2^{k+1},  0, 2^{k+1},0),(-2^{k+1},  0, -2^{k+1},0).
\end{split}
\end{equation*}
Thus, the solutions $(X,Y,W,Z)$ to \eqref{sys1} and \eqref{sys2} are
\begin{equation*}
\label{sol-semibent}
\begin{split}
& (0, \sqrt{2},  0, \sqrt{2}), (0, \sqrt{2},  0, -\sqrt{2}), ( 0,-\sqrt{2},  0,  \sqrt{2}), (0, -\sqrt{2}, 0, -\sqrt{2}),\\
& (\sqrt{2}, 0, \sqrt{2},  0), (\sqrt{2}, 0, -\sqrt{2},  0), (-\sqrt{2}, 0, \sqrt{2},  0), (-\sqrt{2}, 0, -\sqrt{2},  0),
 \end{split}
\end{equation*}
which also satisfy condition~(*), and $(ii)$ is shown.
\qed
\end{proof}

\section{Constructions of generalized bent functions in $\BBZ_8$}

In this section we define several classes of generalized bent Boolean functions.
\begin{thm}
\label{thm:1st-construction}
If $f:\BBZ_2^{n+2}\to \BBZ_8$ ($n$ even) is given by
\[
f(\xx,y,z)=4 c(\xx)+(4a(\xx)+2c(\xx)+1) y+(4b(\xx)+2c(\xx)+1)z-2yz,
\]
 where $a,b,c\in\cB_n$ such that all $a,b,c$, $a\oplus c$, $b\oplus c$ and $a\+b$ are bent
satisfying
\begin{equation}
\label{eq-wc}
W_a(\xx)W_b(\xx)+W_{a\+c}(\xx)W_{b\+c}(\xx)=-2W_{a\+b}(\xx)W_c(\xx)), \text{ for all } \xx\in\BBZ_2^n,
\end{equation}
 then $f$ is gbent in $\cGB_{n+2}^8$.
\end{thm}
\begin{proof}
We compute the Walsh--Hadamard coefficients (using the fact that $\zeta=\frac{1}{\sqrt{2}}(1+\imath)$ and $\zeta^2=\imath$)
\begin{equation}
\begin{split}
&2^{(n+2)/2} \cH_f(\uu,v,w)
=\sum_{(\xx,y,z)\in\BBZ_2^{n+2}} \zeta^{f(\xx,y,z)} (-1)^{\uu\cdot \xx\+vy\+wz}\\
&= \sum_{\xx\in \BBZ_2^n} \zeta^{4c(\xx)} (-1)^{\uu\cdot \xx}\\
&\quad\cdot \sum_{(y,z)\in\BBZ_2^2}
\zeta^{(4a(\xx)+2c(\xx)+1) y+(4b(\xx)+2c(\xx)+1)z-2yz} (-1)^{vy\+wz}\\
&= \sum_{\xx\in \BBZ_2^n}   (-1)^{c(\xx)\oplus\uu\cdot \xx} \cdot
\left(1+ (-1)^{v}(-1)^{a(\xx)}\imath^{c(\xx)} \zeta\right.\\
&\qquad\qquad \left.+(-1)^{w}(-1)^{b(\xx)}\imath^{c(\xx)} \zeta +
 (-1)^{a(\xx)\oplus b(\xx)\+c(\xx)\+v\+w} \right).
\end{split}
\end{equation}
Applying equation~\eqref{eq-zs} with $(z,s)=(\imath ,c(\xx))$, that is,
\begin{eqnarray*}
&& i^{c(\xx)}=\frac{1+(-1)^{c(\xx)}}{2}+\frac{1-(-1)^{c(\xx)}}{2}\imath,
\end{eqnarray*}
we obtain
\begin{eqnarray*}
&&2\cH_f(\uu,v,w)=W_c(\uu)+\frac{(-1)^v \zeta }{2} \left(W_{a\+c}(\uu)+W_{a}(\uu)+\imath W_{a\+c}(\uu)-\imath W_{a}(\uu)\right)\\
&&\quad+ \frac{(-1)^w \zeta}{2} \left(W_{b\+c}(\uu)+W_{b}(\uu)+\imath W_{b\+c}(\uu)-\imath W_{b}(\uu)\right)
+ (-1)^{v\+w} W_{a\+b}(\uu)\\
&&\quad= W_c(\uu)+\frac{(-1)^v}{\sqrt{2}} (W_{a}(\uu)+\imath W_{a\+c}(\uu))+
\frac{(-1)^w}{\sqrt{2}} (W_{b}(\uu)+\imath W_{b\+c}(\uu))\\
&&\qquad\qquad\qquad + (-1)^{v\+w} W_{a\+b}(\uu).
   \end{eqnarray*}

Therefore, the real and the imaginary parts of $cH_f(\uu,v,w)$ are
\begin{eqnarray*}
{Re}(\cH_f(\uu,v,w))&=&  W_c(\uu)+(-1)^{v\+w} W_{a\+b}(\uu)+ \frac{(-1)^vW_{a}(\uu) +(-1)^w W_{b}(\uu)}{\sqrt{2}},\\
{Im}(\cH_f(\uu,v,w))&=& \frac{(-1)^v W_{a\+c}(\uu) +(-1)^w W_{b\+c}(\uu)}{\sqrt{2}}.
   \end{eqnarray*}
   and so,
   {\small
   \begin{equation}
   \begin{split}
   \label{eq-4H}
&4 |\cH_f(\uu,v,w)|^2
= \frac{1}{2}\left(W_{a}(\uu)^2+W_{b}(\uu)^2+W_{a\+c}(\uu)^2+W_{b\+c}(\uu)^2\right.\\
&\qquad\qquad \left.+2 W_c(\uu)^2+2W_{a\+b}(\uu)^2\right)\\
&\qquad+(-1)^{v+w}(W_{a}(\uu) W_{b}(\uu)+W_{a\+c}(\uu) W_{b\+c}(\uu)+2W_{c}(\uu)W_{a\+b}(\uu))\\
&\qquad+\sqrt{2}\left((-1)^v(W_a(\uu)W_c(\uu)+
W_b(\uu)W_{a\+b}(\uu))\right.\\
&\qquad \left. +(-1)^w (W_b(\uu) W_c(\uu)+W_a(\uu)W_{a\+b}(\uu))   \right)
\end{split}
\end{equation}
}
\begin{sloppypar}
Since $a,b,c,a\+c,b\+c,a\+b$ are all bent then
 $|W_{a}(\uu)|=|W_{b}(\uu)|=|W_{c}(\uu)|=|W_{a\+b}(\uu)|=|W_{a\+c}(\uu)|=|W_{b\+c}(\uu)|=1$. Further, from the imposed conditions
 on these functions' Walsh--Hadamard coefficients, we see that
 $W_{a}(\uu) W_{b}(\uu)+W_{a\+c}(\uu) W_{b\+c}(\uu)+2W_{c}(\uu)W_{a\+b}(\uu)=0$, and also
 $W_a(\uu)W_c(\uu)+W_b(\uu)W_{a\+b}(\uu)=0$, $W_b(\uu) W_c(\uu)+W_a(\uu)W_{a\+b}(\uu)=0$
 (that is because if $W_a(\uu)$ and $W_b(\uu)$ have the same sign, then $W_c(\uu), W_{a\+b}$ have opposite signs; further,
 $W_a(\uu)$ and $W_b(\uu)$ have opposite signs, then $W_c(\uu), W_{a\+b}$ have the same sign).
 Using these equations, we get that
 $
 4 |\cH_f(\uu,v,w)|^2=4,
 $
 and so, $f$ is gbent.
 \end{sloppypar}
 \qed
\end{proof}

 \begin{sloppypar}
 \begin{rem}
 It is rather straightforward to see that condition \eqref{eq-wc} has $16$ solutions. More precisely,
 $(W_a(\xx),W_b(\xx),W_{a\+c}(\xx),W_{b\+c},W_{a\+b}(\xx),W_c(\xx))$ could be any of the
 following tuples:
 \[
 \begin{array}{ll}
  (-1,-1,-1,-1,-1,1); & (-1,-1,-1,-1,1,-1);\\
   (-1,1,1,1,-1,1); &  (-1,-1,1,1,1,-1);\\
  (-1,1,-1,1,-1,-1); & (-1,1,-1,1,1,1);\\
 (-1,1,1,-1,1,-1); & (-1,1,1,-1,1,1); \\
 (1,-1,-1,1,-1,-1); & (1,-1,-1,1,1,1);\\
 (1,-1,1,-1,1,-1); & (1,-1,1,-1,1,1);\\
 (1,1,-1,-1,-1,1); & (1,1,-1,-1,1,-1);\\
  (1,1,1,1,-1,1); & (1,1,1,1,1,-1).
   \end{array}
   \]
 \end{rem}
  \end{sloppypar}

\begin{thm}
If $f:\BBZ_2^{n+2}\to \BBZ_8$ ($n$ even) is given by
\begin{equation}
\label{eq:constr2}
f^{\epsilon}(\xx,y,z)=4 c(\xx)+(4a(\xx)+1) y+(4b(\xx) +1)z+2\epsilon yz,
\end{equation}
 where $\epsilon\in\{1,-1\}$, $a,b,c\in\cB_n$ such that all $c$, $a\oplus c$, $b\oplus c$ and $a\+b\+c$ are bent,
 with
 \begin{equation}
 \label{eq-cond-C2}
 \begin{split}
& W_{a\+c}(\uu) W_{b\+c}(\uu)+\epsilon W_{c}(\uu)W_{a\+b\+c}(\uu)=0,\text{ for all } \uu\in\BBZ_2^n,
 \end{split}
 \end{equation}
 then $f$ is gbent in $\cGB_{n+2}^8$.
\end{thm}
\begin{proof}
As in the proof of Theorem~\ref{thm:1st-construction}, we compute the Walsh--Hadamard coefficients, obtaining
\begin{eqnarray*}
2\cH_{f^{\epsilon}}(\uu,v,w)
&=& W_c(\uu)+(-1)^v \zeta W_{a\+c}(\uu)+(-1)^w \zeta W_{b\+c}(\uu)\\
&&\qquad\qquad +(-1)^{v\+w}\zeta^{2+2\epsilon} W_{a\+b\+c}(\uu)\\
&=& W_c(\uu)-\epsilon (-1)^{v\+w} W_{a\+b\+c}(\uu)\\
&&\quad+\frac{(-1)^v W_{a\+c}(\uu)+ (-1)^w  W_{b\+c}(\uu)}{\sqrt{2}}\\
&&\quad+
\imath\frac{(-1)^v W_{a\+c}(\uu)+ (-1)^w  W_{b\+c}(\uu)}{\sqrt{2}},
\end{eqnarray*}
using the fact that $\zeta^{2+2\epsilon}=-\epsilon$, for $\epsilon\in\{1,-1\}$.
Taking the square of the complex norm, we get
\begin{eqnarray*}
&& 4 |\cH_{f^{\epsilon}}(\uu,v,w)|^2=W_{a\+c}(\uu)^2+W_{b\+c}(\uu)^2+W_c(\uu)^2+W_{a\+b\+c}(\uu)^2\\
&&\qquad \qquad\qquad\quad +2(-1)^{v+w}
\left(W_{a\+c}(\uu)W_{b\+c}(\uu)+\epsilon W_{c}(\uu) W_{a\+b\+c}(\uu)\right)\\
&&\qquad \qquad\qquad\quad +\sqrt{2}\left((-1)^v (W_{a\+c}(\uu) W_{c}(\uu)+\epsilon W_{b\+c}(\uu)W_{a\+b\+c}(\uu)) \right.\\
&&\qquad \qquad\qquad\quad \left.+
(-1)^w (W_{b\+c}(\uu)W_{c}(\uu)+ \epsilon W_{a\+c}(\uu)W_{a\+b\+c}(\uu)) \right)\\
&&\qquad \qquad\qquad\quad = 4,
\end{eqnarray*}
because  $c$, $a\oplus c$, $b\oplus c$ and $a\+b\+c$  are all bent, so their Walsh--Hadamard coefficients
are~1 in absolute values, and equation~\eqref{eq-cond-C2} implies that the remaining coefficients are all 0
(that can be seen by the following argument: if $A,B,C,D\in \{\pm 1\}$, and $AB+CD=0$, then by multiplying by $BC$, we get
$AC+BD=0$, and by multiplying by $AC$ we get $BC+AD=0$).

Therefore, $|\cH_{f^{\epsilon}}(\uu,v,w)|^2=1$, so $f$ is gbent, and the theorem is proved.
\qed
\end{proof}
\begin{rem}
It is rather easy to see that equation~\eqref{eq-cond-C2} has $8$ solutions
(as expected, since there are four degrees of freedom and one constraint). Moreover,
one can give plenty of concrete examples of functions $a,b,c$ satisfying the conditions of our theorem. For example,
if $\epsilon=-1$, one could take in equation~\eqref{eq:constr2}, a bent Boolean $c$, and $a=b$ such that $c\oplus a$ is bent (for instance,
if $a=b$  are affine functions, that condition is immediate). Then,
$W_{a\+c}(\uu)W_{b\+c}(\uu)+\epsilon W_{c}(\uu) W_{a\+b\+c}(\uu)=
W_{c\+a}(\uu)^2-W_c(\uu)^2=0$, and  so, $g$ as in our theorem is gbent.
\end{rem}

\begin{thm}
Let $f:\BBZ_2^{n+1}\to \BBZ_8$ ($n$ is even) be given by
\begin{equation}
\label{eq:constr3}
f(\xx,y)=4 c(\xx)+(4a(\xx) +4c(\xx)+2\epsilon) y,
\end{equation}
where $\epsilon\in\{1,-1\}$. Then $f$ is gbent in $\cGB_{n+1}^8$ if and only if $a,c$ are bent in $\cB_n$.
Moreover, if $g$ is given by
\begin{equation}
\label{eq:constr4}
g(\xx,y)=4 c(\xx)+(4a(\xx)+2c(\xx)+2\epsilon) y,
\end{equation}
 where $\epsilon\in\{1,-1\}$, $a, c\in\cB_n$ such that  $a,c$, $a\oplus c$ are all bent,
 then $g$ is gbent in $\cGB_{n+1}^8$. Further, let
  $h$ be given by
\begin{equation}
\label{eq:constr5}
h(\xx,y)=4 c(\xx)+(4a(\xx)+2\epsilon) y,
\end{equation}
 where $\epsilon\in\{1,-1\}$. Then  $h$ is gbent in $\cGB_{n+1}^8$ if and only if $c,a\+c$ are bent in $\cB_n$.
\end{thm}
\begin{proof}
We will show the first claim, since the proof of the remaining ones are absolutely similar.
As in the proof of Theorem~\ref{thm:1st-construction}, the Walsh--Hadamard coefficients at an arbitrary input $(\uu,v)$ are
\begin{eqnarray*}
\sqrt{2}\cH_f(\uu,v)=W_c(\uu)+\imath^\epsilon (-1)^v W_a(\uu)=
W_c(\uu)+\epsilon\, \imath (-1)^v W_a(\uu),
\end{eqnarray*}
and so,
\[
2|\cH_f(\uu,v)|^2=W_c(\uu)^2+W_a(\uu)^2.
\]
If $a,c$ are bent, then $|W_c(\uu)|=|W_a(\uu)|=1$, and so $|\cH_f(\uu,v)|=1$, that is $f$ is gbent.
If $f$ is gbent, then the equation $W_c(\uu)^2+W_a(\uu)^2=2$ has as rational solutions only $|W_c(\uu)|=|W_a(\uu)|=1$, and so, $a,c$ are bent.
\qed
\end{proof}



\begin{thebibliography}{22}



\bibitem{CH1} C.~Carlet, {\em Boolean functions for cryptography and error correcting codes},
In: Y. Crama, P. Hammer  (eds.), Boolean Methods and Models,
Cambridge Univ. Press, Cambridge. Available:
http://www-roc.inria.fr/secret/Claude.Carlet/pubs.html.

\bibitem{CH2} C.~Carlet, {\em Vectorial Boolean functions for cryptography},
In: Y. Crama, P. Hammer  (eds.), Boolean Methods and Models,
Cambridge Univ. Press, Cambridge. Available: http://www-roc.inria.fr/secret/Claude.Carlet/pubs.html.


\bibitem{PS09} T.~W.~Cusick, P.~St\u anic\u a,
Cryptographic Boolean functions and applications, Elsevier -- Academic Press, 2009.



\bibitem{DI75} J.~F.~Dillon,
{\em Elementary Hadamard difference sets},
Proc. of Sixth S.E. Conference of Combinatorics, Graph Theory, and Computing,
Congressus Numerantium No. XIV, Utilitas Math., Winnipeg 1975, 237--249.



\bibitem{GSS11} S.~Gangopadhyay, B.~K.~Singh and P.~St\u anic\u a,
{\em On Generalized Bent Functions}, preprint.

\bibitem{GSS12} S.~Gangopadhyay, B.~K.~Singh and P.~St\u anic\u a,
{\em Generalized Bent Functions in the Framework of
Generalized Partial Spreads}, preprint.

\bibitem{Hi87}
M. D. Hirschhorn, A simple proof of Jacobi's four–square theorem, Proc. Amer. Math.
Soc., 101 (1987), 436--438.

\bibitem{KSW85} P.~V.~Kumar, R.~A.~Scholtz, and L.~R.~Welch,
{\em Generalized bent functions and their properties},
J. Combin. Theory (A) 40 (1985), 90--107.




\bibitem{MWSL}
F.~J.~MacWilliams, N.~J.~A.~Sloane,
\newblock The theory of error--correcting codes,
\newblock North-Holland, Amsterdam, 1977.





\bibitem{RO76} O.~S.~ Rothaus, {\em On bent functions},
J. Combinatorial Theory Ser. A  20 (1976), 300--305.

\bibitem{SM02} P.~Sarkar,  S.~Maitra,
{\em Cross--Correlation Analysis of Cryptographically Useful
Boolean Functions and S-Boxes},
Theory Comput. Systems 35 (2002), 39--57.





\bibitem{KAI07}
K-U.~Schmidt, {\em Quaternary Constant-Amplitude Codes for Multicode CDMA},
IEEE International Symposium on Information Theory, ISIT'2007
(Nice, France, June 24--29, 2007),  2781--2785;
available at http://arxiv.org/abs/cs.IT/0611162.

\bibitem{ST09} P.~Sol\'e, N.~Tokareva,
{\em Connections between Quaternary and Binary Bent Functions},
http://eprint.iacr.org/2009/544.pdf.




\end{thebibliography}
\end{document}